\documentclass[12points,reqno]{amsart}
\usepackage{amsfonts}
\usepackage{amssymb}
\usepackage{graphicx}
\usepackage{amsmath,nicefrac}

\newtheorem{theorem}{Theorem}
\newtheorem{lemma}{Lemma}

\theoremstyle{definition}

\theoremstyle{remark}

\numberwithin{equation}{section}

\everymath{\displaystyle}

\begin{document}
\title{ Simplicity of Lie algebras of Poisson brackets }

%

\author{Adel Alahmadi}
\address{Department of Mathematics, King Abdulaziz University, P.O.Box 80203, Jeddah, 21589, Saudi Arabia}
\email{analahmadi@kau.edu.sa}
\author{Hamed Alsulami}
\address{Department of Mathematics, King Abdulaziz University, P.O.Box 80203, Jeddah, 21589, Saudi Arabia}
\email{hhaalsalmi@kau.edu.sa}
%
%


\maketitle

\begin{abstract}
Let $A$ be an associative commutative algebra with $1$ over a field of zero characteristic, $\{,\} : A \times A \to A$ is a Poisson bracket, $Z = \{ a \in A \mid \{a, A\} = (0) \}.$ We prove that if $A$ is simple as a Poisson algebra then the Lie algebra $\nicefrac{\{A,A\}}{\{A,A\}\cap Z}$ is simple.
\end{abstract}

\maketitle

\section{Introduction}

Let $F$ be a field and let $A$ be an associative commutative $F$- algebra. A bilinear bracket $\{ , \}: A \times A \to A$ is called \underline{Poisson bracket} if
\medskip
\begin{enumerate}
  \item[(1)] $(A, \{ , \})$ is a Lie algebra,
  \item[(2)] $\{ab, c\}=a\{b, c\}+ \{a, c\}b$ for arbitrary elements $a, b, c \in A.$

  An algebra $A$ with a Poisson bracket is called a \underline{Poisson algebra} $(see[4]).$
  The conditions (1), (2) imply the Poisson identity
  \item[(3)] $\{ab, c\}=\{a, bc\}+ \{b, ac\}$ that we will use later.
\end{enumerate}
\medskip

We say that an ideal $I$ of an algebra $A$ is a \underline{Poisson ideal} if $\{I, A\}\subseteq I.$ A Poisson algebra without nontrivial Poisson ideals is called a simple Poisson algebra.
Let $Z=\{a \in A | \{a, A\} =(0)\}.$

\section{ Main Result}

\begin{theorem}\label{thm1} Let $(A, \{,\})$ be a simple Poisson algebra over a field of zero characteristic, $A \ni 1$. Then the Lie algebra $\nicefrac{\{A, A\}}{ \{A, A\}\cap Z}$ is simple.

\end{theorem}
This theorem is an analog of the celebrated theorem of I. Herstein [3].
\medskip

\medskip

\begin{lemma}\label{lem1}
Let $(A, \{, \})$ be a simple Poisson algebra over a field of zero characteristic, $A \ni 1.$ Then the algebra $A$ does not contain a nonzero nilpotent element.
\end{lemma}

\begin{proof}
Let $N(A)$ be the radical of the algebra $A,$ $ N(A)=\{a \in A | \text{ the element } a \text{ is nilpotent} \}.$ It is known (see [2], p.107) that in an algebra over a field of zero characteristic the radical is invariant with respect to all derivations of the algebra $A.$ This implies $\{N(A), A\} \subseteq N(A),$ hence $N(A)$ is a Poisson ideal of $A.$ From the simplicity of the Poisson algebra $A$ it follows that $N(A)=(0)$ or $N(A)=A.$ The second option is not possible since $A \ni 1.$ Hence $N(A)=(0)$, which completes the proof of the lemma.
\end{proof}

\medskip

For a subspace $X \subset A,$ let $id_{A}(X)=X + AX$ be the ideal of $A$ generated by $X.$
Let $U$ be an ideal of the Lie algebra $\{A, A\}.$ Define the descending derived series of ideals: $U^{[0]}=U, U^{[i + 1]}=\{U^{[i]}, U^{[i]} \}.$  The following lemma is an analog of the Lemma $1$ from [1].
\medskip

\begin{lemma}\label{lem2}
$\{id_{A} (U^{[3]}), A\} \subseteq U.$
\end{lemma}

\begin{proof}
For an arbitrary ideal $I$ of the Lie algebra $\{A, A\}$ we have $\{\{I, I\}, A\} \subseteq \{I, \{I, A\}\} \subseteq I.$ Hence $\{U^{[i + 1}, A\}\subseteq U^{[i]}$ for $i \geq 0.$

For arbitrary elements $a, b, u, v \in A$ we have $\{au, v\}=a\{u, v\} + u \{a, v\}.$ Hence $\{\{au, v\}, b\}=\{a\{u, v\}, b\} + \{u \{a, v\}, b\}.$
By the Poisson identity (3) $$\{u \{a, v\}, b\} = \{u, \{a, v\}b\}+\{\{a, v\}, ub\}.$$ This implies $$\{a \{u, v\}, b\} = \{\{au, v\}, b\}-\{u, \{a,v\} b\}- \{\{a, v\},ub \}.$$

The first and the third summands on the right hand side lie in $\{\{v, A\}, A\},$ the second summand lies in $\{u, A\}.$ If $u, v \in U^{[2]}$ then by the remark above $\{v, A\}+\{u, A\}\subseteq U^{[1]}$ and $\{\{v, A\}, A\}\subseteq \{U^{[1]}, A\}\subseteq U^{[0]}=U, \{id_{A} (U^{[3]}), A\}\subseteq U,$ which completes the proof of the lemma.

\end{proof}

\medskip

From now on we assume that $A$ is a simple Poisson algebra over a field of characteristic $0$ and $A \ni 1.$

\medskip

\begin{lemma}\label{lem3}
If $I$ is a nonzero ideal of the algebra $A$ and $\{I, \{A, A\}\}\subseteq I$ then $I=A.$
\end{lemma}

\begin{proof}
Let $\sqrt{I}$ be the set of all elements of the algebra $A$ that are nilpotent module $I.$ Then $\sqrt{I} \unlhd A$ and $\nicefrac{\sqrt{I}}{I}$ is the radical of the algebra $\nicefrac{A}{I}.$

Let $d$ be a derivation of the algebra $A$ such that $d(I) \subseteq I.$ Then $d$ induces a derivation on the factor algebra $\nicefrac{A}{I}.$ Since the radical of algebra $\nicefrac{A}{I}$ is invariant under all derivations of $\nicefrac{A}{I}$ if follows that $d \left(\nicefrac{\sqrt{I}}{I}\right)\subseteq \nicefrac{\sqrt{I}}{I}$ and therefore $d(\sqrt{I})\subseteq \sqrt{I}.$

This implies $\{\sqrt{I}, \{A, A\}\} \subseteq\sqrt{ I}.$ If $\sqrt{I} = A \ni 1$ then $I=A$ as well. From now on considering the ideal $\sqrt{I}$ instead of $I$ we will assume that $\sqrt{I}=I$.

For arbitrary element $a \in A, u \in I$ we have $\{u, \{u, a^2\}\}=2\left(\{u, a\}^2 + \{u, \{u, a\}\} a \right),$ $\{u, \{u, A\}\}\subseteq \{u, \{A, A\}\}\subseteq I.$ Hence $\{u, a\}^2 \in I, \{u, a\} \in \sqrt{I} =I.$ We proved that $\{I, A\}\subseteq I,$ i.e $I$ is a Poisson ideal of $A.$ Since the Poisson algebra$\left(A, \{,\} \right)$ is simple we conclude that $I=A.$ This completes the proof of lemma.
\end{proof}

\medskip

\begin{lemma}\label{lem4}
If $U$ is a proper ideal of the Lie algebra $\{A, A\}$ then $U^{[3]}=0.$
\end{lemma}

\begin{proof}
Let $I=id_A\left(U^{[3]}\right).$ Clearly $\{I,\{A, A\} \}\subseteq I.$ If $U^{[3]}\neq 0$ then by Lemma 3 we have $I=A.$ Hence by Lemma 2 $\{A, A\}\subseteq U.$ This completes the proof of the lemma.
\end{proof}

\medskip

For an element $a \in A$ consider the adjoint operator$ad(a): A \to A, x \to\{a, x\}.$

\medskip

\begin{lemma}\label{lem5}
Let $a \in A, n\geq 1$ and suppose that $ad(a)^n=0.$ Then $ad(a)=0.$
\end{lemma}

\begin{proof}
We will prove that $ad(a)^n=0, n \geq 2,$ implies $ad(a)^{n-1}=0.$ Indeed, let $b \in A.$ Consider the element $b_1=ad(a)^{n-2}b,$ we have $\{a,\{a,b_1\}\}=ad(a)^nb=0.$
Now,
$$0=ad(a)^nb_1^n=n!\{a,b_1\}^n+\sum\limits_{k_1+\cdots+k_n=n\atop (k_1,\cdots,k_n)\neq (1,1,\cdots,1)} (ad(a)^{k_1}b_1)\cdots(ad(a)^{k_n}b_1)$$
In each summaned on the right hand side for at least one $i,\, 1\leq i\leq n,$ we have $k_i\geq 2.$
Hence $ad(a)^{k_i}b_1=0.$ Since the characteristic of the ground field is $0$  we conclude that $\{a,b_1\}^n=0.$
By Lemma $1$ it implies that $\{a,b_1\}=ad(a)^{n-1}b_1=0,$ which completes the proof of the lemma.
\end{proof}

\medskip

\begin{lemma}\label{lem6}
Let $U$ be an abelian ideal of the algebra $\{A,A\}.$ Then $\{U,A\}=(0).$
\end{lemma}

\begin{proof}
For an arbitrary element $u\in U$ we have $$ ad(u)^3A=\{u,\{u,\{u,A\}\}\}\subseteq \{U,\{U,\{A,A\}\}\subseteq \{U,U\}=(0).$$
By Lemma $5$ $ad(u) A=\{u,A\}=(0),$ which completes the proof of the lemma.
\end{proof}

\medskip

\begin{lemma}\label{lem7}
Let $U$ be an ideal of the algebra $\{A,A\}$ such that $U^{[3]}=(0).$ Then $\{U,A\}=(0).$
\end{lemma}

\begin{proof}
The ideal $U^{[2]}$ of the Lie algebra $\{A,A\}$ is abelian. Hence by Lemma $6$ $\{U^{[2]},A\}=(0).$ Let $u\in U^{[1]}.$ we have
$$ad(u)^4 A=\{u,\{u,\{u,\{u,A\}\}\}\}\subseteq \{U^{[1]},\{U^{[1]},U^{[1]}\}\}\subseteq \{A,U^{[2]}\}=(0).$$
Hence from Lemma $5$ it follows that $\{U^{[1]},A\}=(0).$ Again choosing an arbitrary element $u\in U$ we get
$$ad(u)^4 A=\{u,\{u,\{u,\{u,A\}\}\}\}\subseteq \{U,\{U,U\}\}\subseteq\{A,U^{[1]}\}=(0).$$ Hence by Lemma $5$ $\{U,A\}=(0),$ which completes the proof of the lemma.
\end{proof}

\begin{proof}[ Proof of Theorem \ref{thm1}]
Now combining Lemmas $4$ and $7$ we get the assertion of the theorem.
\end{proof}

\end{document}